\documentclass{amsart}

\usepackage{amsthm,amsfonts,amsmath,amssymb,latexsym,epsfig,mathrsfs,yfonts,marvosym}
\usepackage[usenames]{color}

\DeclareMathAlphabet\oldmathcal{OMS}        {cmsy}{b}{n}
\SetMathAlphabet    \oldmathcal{normal}{OMS}{cmsy}{m}{n}
\DeclareMathAlphabet\oldmathbcal{OMS}       {cmsy}{b}{n}

\usepackage{eucal}

\usepackage{graphicx}
\usepackage[all]{xy}
\usepackage{epsfig}

\newtheorem{theorem}{Theorem}[section]
\newtheorem{lemma}[theorem]{Lemma}
\newtheorem{proposition}[theorem]{Proposition}

\newtheorem{remark}{Remark}[section]

\newtheorem{ack}{Acknowledgments\!}

\def\BOne{{\mathchoice {\rm 1\mskip-4mu l} {\rm 1\mskip-4mu l}
                          {\rm 1\mskip-4.5mu l} {\rm 1\mskip-5mu l}}}

\def\fract#1#2{\raise4pt\hbox{$ #1 \atop #2 $}}

\def\bbc{{\mathbb C}}

\def\bbr{{\mathbb R}}

\def\bbz{{\mathbb Z}}

\def\gre{\epsilon}

\def\gri{\iota}
\def\grk{\kappa}
\def\grl{\lambda}

\def\grs{\sigma}
\def\grt{\tau}

\def\grG{\Gamma}

\def\grL{\Lambda}

\def\grS{\Sigma}

\def\bfa{{\bf a}}
\def\bfb{{\bf b}}

\def\bfl{{\bf l}}

\def\bfx{{\bf x}}
\def\bfy{{\bf y}}
\def\bfz{{\bf z}}

\def\cala{{\mathcal A}}

\def\cald{{\mathcal D}}

\def\calf{{\mathcal F}}

\def\cals{{\oldmathcal S}}

\def\la#1{\hbox to #1pc{\leftarrowfill}}
\def\ra#1{\hbox to #1pc{\rightarrowfill}}

\def\ga{{\mathfrak a}}

\def\gc{{\mathfrak c}}

\def\ge{{\mathfrak e}}
\def\gf{{\mathfrak f}}

\def\gh{{\mathfrak h}}
\def\gi{{\mathfrak i}}

\def\gn{{\mathfrak n}}
\def\go{{\mathfrak o}}

\def\gr{{\mathfrak r}}

\def\gt{{\mathfrak t}}
\def\gu{{\mathfrak u}}

\def\gA{{\mathfrak A}}

\def\gC{{\mathfrak C}}
\def\gD{{\mathfrak D}}

\def\gH{{\mathfrak H}}

\def\gN{{\mathfrak N}}

\def\gR{{\mathfrak R}}

\def\gX{{\mathfrak X}}

\def\txi{\tilde{\xi}}
\def\teta{\tilde{\eta}}
\def\tPhi{\tilde{\Phi}}
\def\tg{\tilde{g}}

\def\hook{\mathbin{\hbox to 6pt{%
                 \vrule height0.4pt width5pt depth0pt
                 \kern-.4pt
                 \vrule height6pt width0.4pt depth0pt\hss}}}

\begin{document}
\bibliographystyle{amsalpha}

\title{The Sasakian Geometry of the Heisenberg Group}\thanks{During the preparation of this work the author was partially supported by NSF grant DMS-0504367.}

\author{Charles P. Boyer}
\address{Department of Mathematics and Statistics,
University of New Mexico, Albuquerque, NM 87131.}

\email{cboyer@math.unm.edu} 

\begin{abstract}
In this note I study the Sasakian geometry associated to the standard CR structure on the Heisenberg group, and prove that the Sasaki cone coincides with the set of extremal Sasakian structures. Moreover, the scalar curvature of these extremal metrics is constant if and only if the metric has $\Phi$-sectional curvature $-3.$ I also briefly discuss some relations with the well-know sub-Riemannian geometry of the Heisenberg group as well as the standard Sasakian structure induced on compact quotients.
\end{abstract}

\maketitle

\centerline{\it Dedicated to Professor S. Ianus on the occasion of his 70th Birthday}
\bigskip

\section{Introduction}

Recently the Heisenberg group $\gH^{2n+1}$ has been studied from many viewpoints, in particular, harmonic analysis and probability theory, and an important underlying theme is sub-Riemannian or Carnot-Carath\'eodory geometry \cite{Gro96,Neu96,Tha98,BCDT01,Mon02,CDPT07,CCDG07}. Of course, it arose from quantum physics. However, from the point of view of Riemannian geometry, there is a very natural homogeneous Riemannian metric on $\gH^{2n+1}=\bbr^{2n+1}$. This metric appears to have been discovered over 40 years ago by Sasaki \cite{Sas65} and studied a bit later by Tanno \cite{Tan69a,Tan69b}, although the relation with the Heisenberg group was not noticed at that time. It has constant $\Phi$-sectional curvature equal to $-3$, where $\Phi$ is the endomorphism defining the natural CR structure on $\gH^{2n+1}.$ The relationship with the Heisenberg group has been noted in  \cite{BGM06,BGO06,BG05}, and its appearance in CR spherical geometry was studied further in \cite{Kam06,Dav08}. Indeed, Kamishima has noted the important connection between the CR structure on  $\gH^{2n+1}$ and the Bochner-flat structures on $\bbc^n$ classified by Bryant \cite{Bry01}. From the sub-Riemannian viewpoint it was noted in \cite{CDPT07} that the Carnot-Carath\'eodory metric can be viewed as an anisotropic blow-up of the Sasakian metric.

In \cite{BGS06} it was shown that in the case of the standard CR structure on the sphere $S^{2n+1}$ the Sasaki cone coincides with the set of extremal Sasakian structures, and that the only Sasakian metric of constant scalar curvature is the round sphere metric. In this note I prove a similar result for the case of the Heisenberg group $\gH^{2n+1}$. However, in this case due to the existence of dilations in the CR automorphism group of $\gH^{2n+1}$, the Sasaki cone has only dimension $n$ instead of $n+1.$ Nevertheless, in both cases the result is closely related to Bryant's \cite{Bry01} classification of Bochner-flat metrics, as well as Kamishima's work \cite{Kam06}.

\begin{ack}
{\rm The author would like to thank the Courant Institute of Mathematical Sciences of New York University for its hospitality during which this work was done.
I would also like to thank Jeff Cheeger, Misha Gromov, and John Lott for stimulating conversations and Vestislav Apostolov for an important email communication.}
\end{ack}

\section{Brief Review of Sasakian Geometry}

A Sasakian structure is a particular type of contact metric structure. Recall that a contact structure can be given by a codimension one subbundle $\cald$ of the tangent bundle $TM$ which is as far from being integrable as possible. Alternatively, $\cald$ can be defined as the kernel of a smooth 1-form $\eta$ which satisfies $\eta\wedge (d\eta)^n\neq 0$ everywhere on $M.$ The contact structure $\cald$ only depends on $\eta$ up to a multiple by a nowhere vanishing smooth function. A choice of almost complex structure $J$ on the vector bundle $\cald$ endows $M$ with an almost CR structure $(\cald,J).$ Moreover, since $d\eta$ provides $\cald$ with a symplectic structure, we can demand that $J$ is both {\it compatible} with $d\eta$ and that $d\eta$ is {\it tamed} by $J$ in the sense that $d\eta\circ(J\otimes \BOne)$ is a positive definite symmetric bilinear form on $\cald.$ If we extend $J$ to a smooth endomorphism $\Phi$ of $TM$ by demanding that it annihilate the Reeb vector field $\xi$ of $\eta.$ We then have a canonically defined Riemannian metric by setting $g= d\eta\circ(\Phi\otimes \BOne)\oplus \eta\otimes\eta.$ The quadruple $(\xi,\eta,\Phi,g)$ is called a {\it contact metric structure associated} to the contact structure $\cald$. If the Reeb vector field $\xi$ is a Killing field or equivalently an infinitesimal CR transformation, the contact metric structure $(\xi,\eta,\Phi,g)$ is said to be {\it K-contact}, and if in addition the almost CR structure is integrable, $(\xi,\eta,\Phi,g)$ is a {\it Sasakian} structure. A contact structure $\cald$ which admits a contact form $\eta$ and endomorphism $\Phi$ such that $(\xi,\eta,\Phi,g)$ is (K-contact) Sasakian is said to be of {\it (K-contact) Sasaki type}. We also say that the almost CR structure $(\cald,J)$ satisfying $\cald=\ker\eta$ and $J=\Phi_\cald$ is of {\it (K-contact) Sasaki type}. In \cite{BGS06} the authors developed a theory of extremal Sasakian metrics on compact manifolds. Extremal Sasakian metrics are the critical points of the energy functional which is just the $L^2$-norm of the scalar curvature $s_g$. This coincides with the fact that the transverse metric be an extremal K\"ahler metric. The Euler-Lagrange equations were then shown to be equivalent to the $(1,0)$ gradient vector field $\partial_g^\#s_g$ being transversally holomorphic. Thus, Sasakian metrics of constant scalar curvature are extremal. These include Sasaki-Einstein and more generally Sasaki-eta-Einstein metrics \cite{BGM06}. Since constants are not generally $L^2$ functions on non-compact manifolds, I simply define an {\it extremal Sasakian metric} on a non-compact manifold to be one such that $\partial_g^\#s_g$ is transversally holomorphic.

There is a natural sub-Riemannian geometry associated to any contact manifold \cite{Str86,Gro96}. We construct a metric $d$ on the contact manifold $(M,\cald)$ by choosing a Riemannian metric $g$ and defining the distance $d(p,q)$ between any two points $p,q\in M$ by taking the infimum of the distance with respect to $g$ over all piecewise smooth curves joining $p$ to $q$ whose tangents lie in $\cald$ at every point. The metric $d_\cald$ so constructed is called the {\it Carnot-Carath\'eodory metric} or {\it CC metric} for short, and the metric topology on $M$ coincides with the manifold topology. Moreover, its dependence on the Riemannian metric is mild in the sense that if $d_\cald^i$ are CC metrics computed with respect to the Riemannian metrics $g_i$ for $i=1,2$, then they are bi-Lipschitz equivalent on compact subsets of $M$ \cite{Gro96}. When $g$ is the Riemannian metric of a contact metric structure $(\xi,\eta,\Phi,g)$, the transverse metric $g_T=g|_{\cald\times\cald}$ is also called a {\it sub-Riemannian metric}, and if $d_g$ denotes the distance with respect to the Riemannian metric $g$, then we have $d_g\leq d_\cald$ where $d_\cald$ is the sub-Riemannian distance with respect to $g_T.$ So the transverse metric $g_T$ plays two distinct roles, one as a Riemannian metric on the transverse space, and second as a sub-Riemannian metric on all of $M.$ Furthermore, beginning with a contact metric $g=g_T+\eta\otimes \eta$, one can form the {\it anisotropic blow-up} of $g$ by forming the family of ``penalized metrics'' $g_L=g_T+L\eta\otimes\eta$ for some constant $L.$ We can then take the limit $L\rightarrow \infty$ in an appropriate sense. The length becomes squashed in the vertical direction, and the metric spaces $d_{g_L}$ converge in the pointed (fixed base point) Gromov-Hausdorff topology to the CC metric $d_\cald$\footnote{A proof of this in the case of  the Heisenberg group is given in \cite{CDPT07}, but it holds generally for contact metric structures as explained to me by Misha Gromov.}. The small CC boxes \cite{Gro96} have length $\gre$ in each horizontal direction and $\gre^2$ in the vertical dimension; thus, the Hausdorff dimension of a contact manifold with its CC measure is $2n+2.$ A key ingredient in the construction of Carnot-Carath\'eodory metrics is the so-called {\it bracket generating} condition that for any local frame $\{X_1,\cdots,X_{2n}\}$ of $\cald$ the iterated brackets $\{X_i,[X_i,X_j],[X_i,[X_j,X_k]],\cdots\}$ span the whole tangent space at each point. In the case of a contact manifold this is {\it two-step} or {\it strongly} bracket generating, that is, $\{X_i,[X_i,X_j]\}$ span $T_pM$ at all points. A well-known theorem of Chow says that if a sub-Riemannian manifold is bracket generating then any two points can be joined by a curve whose tangent vectors lie in the horizontal subspace at all point along the curve (assuming $M$ is connected, of course).

\section{The Groups of Sasakian and CR Automorphisms}

For any CR structure $(\cald,J)$ on a manifold $M$ we can define the group of CR transformations
\begin{equation}\label{crtrans}
\gC\gR(\cald,J)=\{\phi\in \gD\gi\gf\gf(M)\; | \; \phi_*\cald\subset \cald,~ \phi_*J=J\phi_*\}\, .
\end{equation}
Its Lie algebra $\gc\gr(\cald,J)$ 
can be characterized as
\begin{equation}\label{infcrtrans}
\gc\gr(\cald,J)=\{X\in \gX(M)\; | \; [X,\cald]\subset \cald,~ \pounds_XJ=0\}\, .
\end{equation}
When the CR structure is strictly pseudoconvex the group $\gC\gR(\cald,J)$ is a Lie group \cite{ChMo74,BRWZ04}. 

\begin{lemma}\label{coorient}
If $(\cald,J)$ is a strictly pseudoconvex almost CR structure, the group $\gC\gR(\cald,J)$ preserves the orientation of the bundle $\cald.$
\end{lemma}

\begin{proof}
Suppose that $\phi\in\gC\gR(\cald,J)$ reverses the orientation of $\cald$, i.e. $\phi^*\eta=-f_\phi\eta$ where $f_\phi>0$ and $J^\phi=\phi^{-1}_*J\phi_*=J.$ We show that the Levi form changes sign to give a contradiction. We have 
$$\phi^*(d\eta\circ(J\otimes \BOne))=-d(f_\phi\eta)\circ(J^\phi\otimes \BOne)=-f_\phi d\eta\circ(J\otimes \BOne).$$
\end{proof}

Since $(\cald,J)$ is strictly pseudoconvex, a choice of contact 1-form $\eta$ defines a contact metric structure $\cals=(\xi,\eta,\Phi,g)$ with $\xi$ the Reeb vector field, $J=\Phi|_\cald,$ and $g=d\eta\circ(\Phi\otimes \BOne)\oplus \eta\otimes\eta.$ The automorphism group $\gA\gu\gt(\cals)$ of $\cals$ is the subgroup of $\gC\gR(\cald,J)$ that leaves $\cals$ invariant. In fact it is easy to see that 
$$\gA\gu\gt(\cals)=\{\phi\in \gC\gR(\cald,J)~|~ \phi^*\eta=\eta\}.$$
If the CR structure is of Sasaki type, there is a choice of $\eta$ such that the contact metric structure is Sasakian. It is this situation that interests us.

The Sasaki cone was defined in \cite{BGS06} (Definition 6.7) to be the moduli space of Sasakian structures that are compatible with a given strictly pseudoconvex CR structure $(\cald,J);$ however, it is often convenient to loosen this definition a bit, and think of the Sasaki cone as a kind of pre-moduli space, that is, before modding out by discrete groups. This is an abuse of notation that we have used in \cite{BGS06}. In order to distinguish the two, we shall refer to the latter as the {\it unreduced Sasaki cone}. First, we mention that the construction clearly works equally well for K-contact structures, and goes as follows: fix an almost CR structure of K-contact type, $(\cald,J)$, and choose $\eta$ so that $(\xi,\eta,\Phi,g)$ is K-contact. Define $\gc\gr^+(\cald,J)$ to be the subset of the Lie algebra $\gc\gr(\cald,J)$ defined by
\begin{equation}\label{cr+}
\gc\gr^+(\cald,J)=\{\xi'\in \gc\gr(\cald,J)~|~\eta(\xi')>0\}.
\end{equation}
The set $\gc\gr^+(\cald,J)$ is a convex cone in $\gc\gr(\cald,J)$, and it is open when $M$ is compact\footnote{Since \cite{BGS06} deals almost exclusively with compact manifolds, the compactness assumption was not stated explicitly in Lemma 6.4 of \cite{BGS06}. As we shall see the set $\gc\gr^+(\cald,J)$ is not necessarily open when $M$ is non-compact.}. Furthermore, $\gc\gr^+(\cald,J)$ is invariant under the adjoint action of $\gC\gR(\cald,J)$, and is bijective to the set of K-contact structures compatible with $(\cald,J),$ namely

\begin{equation}
{\mathcal S}(\cald,J)=\left\{
\begin{array}{c}
 \cals=(\xi,\eta,\Phi,
g):\; \cals \; {\rm a~ \text{K-contact}~ structure} \\
({\rm ker}\, \eta, \Phi \mid_{{\rm ker}\, \eta})=
(\cald,J)\end{array} \right\},
\end{equation}
since a contact metric structure $\cals=(\xi,\eta,\Phi,g)$ with underlying almost CR structure $(\cald,J)$ is K-contact if and only if $\xi\in\gc\gr(\cald,J).$ Now fix a K-contact structure $\cals$ and let $\ga$ be a maximal Abelian subalgebra of $\gc\gr(\cald,J)$ containing $\xi.$ We then say that the maximal Abelian subalgebra is of {\it Reeb type} \cite{BG00b}. When both $M$ and $\gC\gR(\cald,J)$ are compact $\ga$ is unique up to conjugacy; however, when either $M$ or $\gC\gR(\cald,J)$ is non-compact there can be many conjugacy classes of maximal Abelian subalgebras. In the case of the sphere the group $\gC\gR(\cald,J)$ is non-compact, but the positivity requirement determined the Sasaki cone \cite{BGS06}. In the non-compact case we shall make a maximal positivity requirement. For any Abelian subalgebra $\ga\subset \gc\gr(\cald,J)$ we set $\ga^+=\ga\cap \gc\gr^+(\cald,J)$. We say that the subset $\ga^+$ is {\it maximal} if whenever $\ga^+\subset \ga_1^+$ for some other Abelian subalgebra $\ga_1$ we have $\ga^+=\ga_1^+.$ I now define an  {\it unreduced Sasaki cone} to be a maximal $\ga^+.$

As in \cite{BGS06} we define the {\it (reduced) Sasaki cone} to be the moduli space of Sasakian (K-contact) structures associated to the strictly pseudoconvex (almost) CR structure $(\cald,J)$ and denote it  by 
$$\grk(\cald,J)={\mathcal S}(\cald,J)/\gC\gR(\cald,J).$$ 
When the action of $\gC\gR(\cald,J)$ on ${\mathcal S}(\cald,J)$ is proper $\grk(\cald,J)$ is always a well-defined Hausdorff space.
I remark that the dimension of $\ga^+$ does not necessarily equal the dimension of $\ga;$ however, we have

\begin{proposition}
Let $\ga^+$ be the Sasaki cone of a contact structure $\cald$ of K-contact type. Let $\grk,k$ denote the  dimensions of $\ga^+$ and $\ga$, respectively. Then  $1\leq \grk\leq k,$ and $\grk=k$ if $M$ is compact. 
\end{proposition}

\begin{proof}
The first statement is immediate from the definition of $\ga^+$ and the fact that it contains at least a ray of Reeb vector fields. The last statement follows from the fact that $\ga^+$ is open in $\ga$ when $M$ is compact. 
\end{proof}

The following generalizes a result of \cite{BGS06} to the case when $M$ is not necessarily compact:

\begin{proposition}\label{properprop}
Let $(M,\cald)$ be a contact manifold of Sasaki type with CR structure $(\cald,J).$ Suppose that $(M,\cald,J)$ is not CR equivalent to the Heisenberg group $\gH^{2n+1}=\bbr^{2n+1}$ nor the sphere $S^{2n+1}$ with their standard CR structures. Then there exists a Sasakian structure $\cals=(\xi,\eta,\Phi,g)$ with $(\ker\eta,\Phi |_\cald)=(\cald,J)$ such that $\gA\gu\gt(\cals)=\gC\gR(\cald,J).$
\end{proposition}

\begin{proof}
Clearly for any Sasakian structure $\cals$ with underlying CR structure $(\cald,J)$ we have $\gA\gu\gt(\cals)\subset \gC\gR(\cald,J),$ so it suffices to prove the inclusion $\gC\gR(\cald,J)\subset \gA\gu\gt(\cals)$ for some Sasakian structure $\cals$. For notational convenience we set $G=\gC\gR(\cald,J).$ The case of $M$ compact is Proposition 4.4 of \cite{BGS06}, so we confine ourselves to the non-compact case. By a theorem of Schoen \cite{Sch95} $G$ acts properly on $M$, so we can use a slice theorem as done in Lemma 2.6 of Lerman \cite{Ler02a} to obtain a 1-form $\eta$ with $\ker\eta=\cald$ such that $G\subset \gC\go\gn(M,\eta).$ Since $(M,\cald)$ is of Sasaki type, there exists a Sasakian structure $\tilde{\cals}=(\txi,\teta,\tPhi,\tg)$ with underlying CR structure $(\cald,J).$ Then $\cals^\phi=(\phi^{-1}_*\txi,\phi^*\teta,\phi^{-1}_*\tPhi\phi_*,\phi^*\tg)$ is Sasakian for all $\phi\in G.$ Now locally we can use the slice theorem to write $M=G\times_KV$ for some compact subgroup $K\subset G,$ and some representation $K\rightarrow GL(V)$ on the vector space $V.$ The action of $K$ on $G\times V$ is given by $(g,v)\mapsto (gk^{-1},k\cdot v),$ so we can average the restriction $\teta|_V$ over $K$ and extend it to all of $M$ by $G$-invariance. This gives a $G$-invariant 1-form $\eta$ belonging to a $G$-invariant Sasakian structure $\cals=(\xi,\eta,\Phi,g).$ So $G\subset \gA\gu\gt(\cals).$
\end{proof}

\begin{remark}\label{dilrem} {\rm The theorem does not apply when $M$ is the Heisenberg group $\gH^{2n+1}$. In this case $\gC\gR(\cald,J)$ contains the  dilation $(x_i,y_i,z)\mapsto (\grl x_i,\grl y_i,\grl^2z)$ which stabilizes the origin $(0,\cdots,0)$ in $\bbr^{2n+1}$. So the action of $\gC\gR(\cald,J)$ is not proper}.
\end{remark}

\section{The Heisenberg Group}

The $2n+1$ dimensional Heisenberg group $\gH^{2n+1}$ can be described by the group of $n+2$ by $n+2$ real matrices of the form
\begin{equation}\label{Heisgroup} 
{H(n)=\left\{\left(
\begin{array}{ccc}
1 &\bfx^t & z\\
0 & 1 & \bfy\\
0 & 0 & 1 \end{array}\right)\  ;\
\bfx,\bfy\in\bbr^n,z\in\bbr\right\}},
\end{equation}
As a manifold it is just $\bbr^{2n+1}.$
There are two natural isomorphic Sasakian structures on $\gH^{2n+1}$: the right invariant contact 1-form is $\eta^R=dz-\bfy\cdot d\bfx,$ and the left invariant contact form is $\eta^L=dz-\bfx\cdot d\bfy.$ These are related by the involution $\gri:\bbr^{2n+1}\ra{1.5} \bbr^{2n+1}$ defined by $\gri(\bfx,\bfy,z)=(\bfy,\bfx,z),$ that is, $\gri^*\eta^L=\eta^R.$ Notice that $\gri$ preserves orientation if $n$ is even, and reverses orientation if $n$ is odd. These contact forms give rise to the right invariant Sasakian structure $\cals^R=(\xi,\eta^R,\Phi^R,g^R)$ and left invariant Sasakian structure $\cals^L=(\xi,\eta^L,\Phi^L,g^L),$ where 
\begin{equation}\label{rsas}
\xi=\partial_z,~\Phi^R= \sum_i[(\partial_{x_i}+y_i\partial_z)\otimes dy_i -\partial_{y_i}\otimes dx_i],~g^R=d\bfx\cdot d\bfx+d\bfy\cdot d\bfy+(dz-\bfy\cdot d\bfx)^2,
\end{equation}
and 
\begin{equation}\label{lsas}
\xi=\partial_z,~\Phi^L= \sum_i[(\partial_{y_i}+x_i\partial_z)\otimes dx_i -\partial_{x_i}\otimes dy_i],~g^L=d\bfx\cdot d\bfx+d\bfy\cdot d\bfy+(dz-\bfx\cdot d\bfy)^2.
\end{equation}
Both Sasakian structures have the same Reeb vector field.
The corresponding contact bundles $\cald^R=\ker \eta^R$ and $\cald^L=\ker \eta^L$ are spanned by $\{U_i^R=\partial_{y_i},V_i^R=(\partial_{x_i}+y_i\partial_z)\}_{i=1}^n$ and $\{U_i^L=\partial_{x_i},V_i^L=(\partial_{y_i}+x_i\partial_z)\}_{i=1}^n,$ respectively. So we have isomorphic underlying CR structures $(\cald^R,J^R),$ and $(\cald^L,J^L)$ where as usual $J=\Phi|_\cald.$ Notice that $V_i^R,U_i^R,\xi$ span the Lie algebra $\gh_{2n+1}$ of the Heisenberg group. However, it is the Lie algebra associated with the {\it left} action, not the right. Likewise, $V_i^L,U_i^L,\xi$ span the Lie algebra $\gh_{2n+1}$ associated to the {\it right} action. So the Heisenberg group has what we can call a {\it bi-Sasakian} structure.

Let $(M,\cals)$ be a Sasakian manifold and let $\phi:M\ra{1.5} M$ be a diffeomorphism, then $(M,\cals^\phi)$ is an isomorphic Sasakian structure where $\phi^*\cals:=(\phi^{-1}_*\xi,\phi^*\eta,\phi^{-1}_*\Phi\phi_*,\phi^*g).$ Then we see that $\cals^L=\gri^*\cals^R.$ Now let us fix one of these structures, namely, the right Sasakian structure $\cals^R$ with its underlying CR structure $(\cald^R,J^R).$ However, for ease of notation I will often drop the superscript $R,$ and refer to this as the {\it standard} Sasakian or CR structure on $\gH^{2n+1}.$ There is also another frequently used model of the standard CR structure on $\gH^{2n+1}$ which is intermediate to the right and left ones, but explicitly exhibits the complex structure on the transverse space. It is defined by the contact 1-form $\eta=dz-2\sum_j(\bfx\cdot d\bfy-\bfy\cdot d\bfx),$ and gives an equivalent Sasakian structure which is invariant under the group action $(\bfz,z)\cdot (\boldsymbol{\zeta},c)=(\bfz +\boldsymbol{\zeta},z+c+2{\rm Im}(\bar{\boldsymbol{\zeta}}\cdot \bfz)$.

\section{The Group of CR Transformations of $\gH^{2n+1}$}

The group of CR automorphisms of the standard strictly pseudoconvex CR structure on $\gH^{2n+1}$ is well-known \cite{Tol78,Fol89}.
Actually, we compute the Lie algebra $\gc\gr(\gH^{2n+1};\cald,J)$. Of course, this gives the component $\gC\gR(\gH^{2n+1};\cald,J)_0$ of $\gC\gR(\gH^{2n+1};\cald,J)$  connected to the identity. We have

\begin{theorem}\label{HeisCR}
The Lie algebra $\gc\gr(\gH^{2n+1};\cald,J)$ of infinitesimal CR transformations for the standard (right) CR structure on $\gH^{2n+1}$ is spanned by the vector fields:
$$\xi=\partial_z,~R_i=\partial_{x_i},~S_i=\partial_{y_i}+x_i\partial_z,~X_{ij}=x_i\partial_{y_j}+ x_j\partial_{y_i} -y_i\partial_{x_j} -y_j\partial_{x_i}+(x_ix_j-y_iy_j)\partial_z$$ 
$$Y_{ij}=x_i\partial_{x_j}-x_j\partial_{x_i}+y_i\partial_{y_j}-y_j\partial_{y_i}, ~D=2z\partial_z+ \sum_k(x_k\partial_{x_k}+y_k\partial_{y_k}).$$ 
\end{theorem}

\begin{proof}
The procedure is well known. First, since the vector fields leave $\eta^R$ invariant they take the form 
\begin{equation}\label{infcon2}
X=\biggr(F-\sum_{i=1}^n y_i\frac{\partial F}{\partial
y_i}\biggl)\frac{\partial}{\partial z} -\sum_{i=1}^n
\frac{\partial F}{\partial y_i}\frac{\partial}{\partial x_i}
+\sum_{i=1}^n\biggl(y_i\frac{\partial F}{\partial z}+
\frac{\partial F}{\partial x_i}\biggl) \frac{\partial}{\partial
y_i}\,,
\end{equation}
where $F$ is an arbitrary smooth function on $\bbr^{2n+1}$ known as a {\it
Hamiltonian} function for the infinitesimal contact transformation
$X.$ We then demand that $X$ satisfy $\pounds_X\Phi^R=0$ where $\Phi^R$ is given by Equation (\ref{rsas}). This gives a system of PDE's for $F$ which is straightforward to solve.

\end{proof}

Note that $(R_i,S_j,\xi)$ span the Lie algebra $\gh^{2n+1}$ of the Heisenberg group and that $(R_i,S_i)=(U_i^L,V_i^L).$ Moreover, $(X_{ij},Y_{ij})$ span the Lie algebra $\gu(n)$ of the unitary group $U(n).$ The Heisenberg algebra $\gh^{2n+1}$ is an ideal in $\gc\gr(\gH^{2n+1};\cald,J)$ with quotient algebra $\gu(n)\oplus \bbr$ where $\bbr$ is generated by $D.$ Summarizing we have

\begin{lemma}\label{crheis}
There is an exact sequence of Lie algebras
$$0\ra{1.8}\gh^{2n+1}\ra{1.8} \gc\gr(\gH^{2n+1};\cald,J)\ra{1.8} \gu(n)\oplus \bbr \ra{1.8} 0.$$
\end{lemma}

From this we have

\begin{theorem}\label{CRHeisgrp}
The group $\gC\gR(\gH^{2n+1};\cald,J)_0$ is isomorphic to the semi-direct product $(U(n)\times \bbr^+)\ltimes \gH^{2n+1}.$ Moreover, the connected component of the group of automorphisms, $\gA\gu\gt(\cals)_0$ of the standard Sasakian structure on $\gH^{2n+1}$, is isomorphic to the semi-direct product $U(n)\ltimes \gH^{2n+1}.$
\end{theorem}

\begin{proof}
By Lemma \ref{crheis} it suffices to show that the vector fields of Theorem \ref{HeisCR} are complete. Since linear vector fields on $\bbr^{2n+1}$ are complete, we only need check this for $X_{ij}.$ Since these vector fields all have the same form, it suffices to look at one case, for example $X_{12}=x_1\partial_{y_2}+ x_2\partial_{y_1} -y_1\partial_{x_2} -y_2\partial_{x_1}+(x_1x_2-y_1y_2)\partial_z.$ The integral curves satisfy
$$\frac{dy_2}{dt}=x_1,~\frac{dy_1}{dt}=x_2,\frac{dx_2}{dt}=-y_1,\frac{dx_1}{dt}=-y_2, ~\frac{dz}{dt}=x_1x_2-y_1y_2.$$
The solutions are sines and cosines and thus defined for all $t.$ This proves the first statement. The only element of $\gc\gr(\gH^{2n+1};\cald,J)$ that does not leave $\cals^R$ invariant is $D$ which generates the dilation group $\bbr^+$ which proves the second statement. 
\end{proof}

Although I have not determined the full automorphism group $\gA\gu\gt(\cals),$ it is easy to see that  there are some discrete elements that are not in $\gA\gu\gt(\cals)_0$. For each $i=1,\cdots, n$ the map $\grs_i$ defined by sending $(x_i,y_i)$ to $(-x_i,-y_i)$ and leaving all other coordinates unchanged is in $\gA\gu\gt(\cals)\setminus \gA\gu\gt(\cals)_0.$ However, the product $\grs_i\grs_j$ is in $\gA\gu\gt(\cals)_0,$ since it can be realized by a rotation generated by $Y_{ij}.$

Now Schoen \cite{Sch95} proved that the group $\gC\gR(\gH^{2n+1};\cald,J)$ acts properly on a strictly pseudoconvex CR manifold $M$ unless $M$ is the Heisenberg group $\gH^{2n+1}$ or the sphere $S^{2n+1}$ with their standard CR structures. In the case of $\gH^{2n+1}$ we see that the group $\gC\gR(\gH^{2n+1};\cald,J)$ does not act properly, since the isotropy subgroup of the origin is the dilatation group $\bbr^+$. It can be shown, however, that the group $\gA\gu\gt(\cals)_0$ does act properly on $\gH^{2n+1}.$

\section{The Sasaki and Extremal Cones of $(\gH^{2n+1};\cald,J)$}

The maximal dimension of a maximal Abelian subalgebra of $\gc\gr(\gH^{2n+1};\cald,J)$ is $n+1$. This can be seen from the decomposition $\gc\gr(\gH^{2n+1};\cald,J)=(\gu(n)\oplus \bbr)\ltimes \gh^{2n+1}$.

\begin{lemma}\label{sasconeHeis}
The unreduced Sasaki cone $\ga^+$ of $\gH^{2n+1}$ is determined by the maximal Abelian subalgebra $\ga$ spanned by the basis $\{\xi,X_{11},\cdots,X_{nn}\}$, and is given by $a>0$ and $b_i\leq 0$ where $(a,b_1,\cdots,b_n)$ are coordinates of $\ga$ with respect to this basis. 
\end{lemma}

\begin{proof}
To satisfy positivity $\ga$ must contain $\xi,$ and any other element of $\gh^{2n+1}$ will not be positive. Moreover, any maximal Abelian subalgebra of $\gu(n)$ is conjugate to that spanned by $\{X_{11},\cdots,X_{nn}\}$, so up to conjugacy we obtain the Abelian algebra $\ga.$ It is easy to check that this is maximal. Then on the algebra $\ga$ we have
\begin{equation}\label{Reebvect}
0<\eta(a\xi+\sum_ib_iX_{ii})=a-\frac{1}{2}\sum_ib_i\bigl(x_i^2+y_i^2\bigr),
\end{equation}
which is positive on all of $\gH^{2n+1}$ if and only if $a>0$ and $b_i\leq 0.$ Furthermore, it can be checked that, up to conjugacy, $\ga_1^+$ of any other maximal Abelian subalgebra $\ga_1$ is contained in $\ga^+.$ 
\end{proof}

As mentioned previously, we have an identification $a^+={\mathcal S}(\cald,J),$ and we denote an element of ${\mathcal S}(\cald,J)$ by $\cals_{a,\bfb}$ where $\bfb=(b_1,\cdots,b_n)$ denote the coordinates of a maximal Abelian subalgebra $\ga$ of $\gu(n).$  From the action of the dilatation subgroup of $\gC\gR(\gH^{2n+1};\cald,J)$ we can put $a=1$, and by the action of $U(n)$ we choose a maximal torus $T^n.$ The normalizer of $T^n$ is the Weyl group of $U(n)$ which is isomorphic to the symmetric group $\grS_n$, so any element of ${\mathcal S}(\cald,J)$ can be brought to the form $\cals_{1,\bfb}$ with the ordered weights. So the Sasaki cone $\grk(\gH^{2n+1};\cald,J)={\mathcal S}(\cald,J)/\gC\gR(\gH^{2n+1};\cald,J)$ associated to the standard CR structure $(\cald,J)$ on $\gH^{2n+1}$ is the subset of $\bbr^n$ defined by $b_n\leq \cdots\leq b_1\leq 0$. It is convenient to set $a_i=-\frac{1}{2}b_i.$ We are now ready for 

\begin{theorem}\label{sasex}
Let $(\cald,J)$ be the standard CR structure on the Heisenberg group $\gH^{2n+1}$, and let $\grk(\cald,J)$ and $\ge(\cald,J)$ denote the Sasaki cone and the extremal Sasaki set on $\gH^{2n+1}$, respectively. Then 
\begin{enumerate}
\item the Sasaki cone $\grk(\cald,J)$ can be identified with set of non-negative ordered $n$-tuples $0\leq a_1\leq \cdots\leq a_n$;  
\item the standard Sasakian structure $\cals_{1,0}$ is the only one with constant scalar curvature, and is  {\it strongly extremal} in the sense of \cite{BGS07b}, and it has a null eta-Einstein structure with constant $\Phi$-sectional curve $-3$;
\item for $\bfa\neq (0,\cdots,0)$ the Sasakian metrics in $\cals_{1,\bfa}$ are complete and extremal, but they do not have constant scalar curvature;
\item  $\grk(\cald,J)=\ge(\cald,J)$;
\item for generic $\bfa$, the automorphism group $\gA\gu\gt(\cals_{1,\bfa})$ is the $n$-torus $T^n.$

\end{enumerate}
\end{theorem}

\begin{proof}
Let us consider the vector field $\xi_\bfa=\xi-2\sum_ia_iX_{ii}$ as a Reeb field. The corresponding contact form is then 
\begin{equation}\label{eta_a}
\eta_\bfa=\displaystyle{\frac{\eta}{1+\sum_ia_i(x_i^2+y_i^2)}}.
\end{equation}
As discussed previously it is easy to see that the vector field $\xi_\bfa$ is complete and nowhere vanishing on $\gH^{2n+1}.$

\begin{lemma}
The flow of $\xi_\bfa$ generates a free proper action $\cala$ of $\bbr$ on $\gH^{2n+1}$ and the quotient space $\gH^{2n+1}/\cala(\bbr)$ is biholomorphic to $\bbc^n.$
\end{lemma}

\begin{proof}
 The integral curves of $\xi_\bfa$ are easily seen to be 
$$x_i(t)=x_i(0)\cos 4a_it+y_i(0)\sin 4a_it,\qquad y_i(t)=y_i(0)\cos 4a_it -x_i(0)\sin 4a_it,$$
$$z(t)=t +z(0) +\sum_i2a_i[x_i(0)y_i(0)\sin^2 4a_it +\frac{x_i(0)^2-y_i(0)^2}{4}\sin 8a_it].$$
If not all $a_i$s vanish the curves are helices, and isomorphic to $\bbr$; whereas, if $a_i=0$ for all $i$ the curve is the $z$ axis. In any case the group law holds and one has an action of $\bbr$ which
one easily sees is both free and proper. Moreover, since $\gH^{2n+1}$ is diffeomorphic to $\bbr^{2n+1}$, the quotient is diffeomorphic to $\bbr^{2n}.$ Since $\Phi_\bfa$ is invariant under the flow of $\xi_\bfa$ and $\Phi_a|_\cald=J,$ almost complex structure induced on the quotient is integrable.

\end{proof}

Continuing with the proof of Theorem \ref{sasex} we notice that the standard CR structure has vanishing Chern-Moser tensor \cite{ChMo74}, and by a result of Webster \cite{Web77} this means that the submersed metrics on the quotient $\bbc^n$ are Bochner-flat. So if we can show that they are complete, we can identify them with Bryant's complete Bochner-flat metrics on $\bbc^n$ \cite{Bry01}. This was done in detail by Kamashima \cite{Kam06}. He uses sub-Riemannian geometry on $\gH^{2n+1}$ identifying the closure of CC-balls with compact subsets of the one point compactification $S^{2n+1}$. We refer to Proposition 3.5 of  \cite{Kam06} for details. Since the Reeb orbits are complete geodesics with respect to the Sasakian metrics $g_\bfa$, it follows that these metrics are complete. Moreover, since Bryant gives a classification of complete Bochner-flat metrics on $\bbc^n$, we obtain all the Sasakian structures in the Sasaki cone $\grk(\cald,J).$ Also as noted by Bryant a result of Abreu \cite{Abr98} implies that Bochner-flat metrics are extremal in the sense of Calabi, so the Sasakian metrics $g_\bfa$ are also extremal which implies (4). This can be seen explicitly once we compute the scalar curvature (see Remark \ref{momrem} below).

The isotropy subgroup of a Sasakian structure $\cals_{1,\bfa}\in\grk(\cald,J)$ is $\gA\gu\gt(\cals_{1,\bfa})$, and one easily sees that Lemma 6.6 of \cite{BGS06} applies to our case, so (5) follows. (This is also noted in \cite{Bry01} for the Bochner-flat metrics on $\bbc^n$, and Kamashima \cite{Kam06} essentially computes $\gA\gu\gt(\cals_{1,\bfa})$ in all cases.) It remains to show that the only structure $\cals_{1,\bfa}$ with constant scalar curvature $s_{g_\bfa}$ is the standard Sasakian structure $\cals_{1,0}.$ We do so by computing $s_{g_\bfa}$ explicitly.

\begin{lemma}\label{scalarcurv}
The scalar curvature $s_{g_\bfa}$ of the Sasakian structures $\cals_{1,\bfa}$ is given by
$$s_{g_\bfa}=-n(2n+7)\frac{\sum a_j^2|z_j|^2}{1+\sum a_i|z_i|^2}+2n(4|\bfa|-1)$$
where $|\bfa|=\sum_ia_i.$
\end{lemma}

\begin{proof}
First we notice that $s_{g_\bfa}$ is related to the scalar curvature of the transverse metric $g^T_\bfa$ by $s_{g_\bfa}=s_{g_\bfa}^T-2n$ where $g_\bfa=g^T_\bfa+\eta_\bfa\otimes \eta_\bfa,$ and we see from Equation (\ref{eta_a}) that $g^T_\bfa= f_\bfa g^T_0$ where $g^T_0$ is the flat transverse metric of the standard Sasakian structure and $f_\bfa=(1+\sum_ja_j|z_j|^2)^{-1}$. Since $f_\bfa$ is basic and $g^T_\bfa=f_\bfa g^T_0$, we can compute $s_{g_\bfa}^T$ from $s_{g^T_0}=0$ by using the well-known formula for how the scalar curvature changes under conformal changes of metrics \cite{Bes}.
\end{proof}

Clearly from this lemma $s_{g_\bfa}$ is constant if and only if $\bfa=0$ which completes the proof of Theorem \ref{sasex}.

\end{proof}

\begin{remark}\label{momrem} 
{\rm The functions that appear in the expression for $s_{g_\bfa}$ are the components of the moment map $\mu_\bfa:\gH^{2n+1}\ra{1.5} \gt_n^*$ defined by $<\mu_a,\grt>=\eta_\bfa(\grt)$ associated to the action of the maximal torus $T^n$ in $\gC\gR(\cald,J)$ whose Lie algebra is spanned by $\{X_{ii}\}_i$. Indeed, with respect to the standard basis of $\gt_n^*=\gt_n=\bbr^n$ we see that the moment map is $\mu_\bfa(z,\bfz,\bar{\bfz})=(h_1,\cdots,h_n)$ with  components 
$$h_i(\bfz,\bar{\bfz})=\eta_\bfa(X_{ii})=-2\frac{|z_i|^2|}{1+\sum_ja_j|z_j|^2},$$
and the scalar curvature $s_{g_\bfa}$ is an affine function of these components, namely
$$s_{g_\bfa}=2n(4|\bfa|-1)+\frac{n(2n+7)}{2}\sum_ia_i^2h_i(\bfz,\bar{\bfz}).$$
Note that although this structure is not toric in the Sasakian sense due to the non-compactness, the transverse geometry is toric, and we can use the result of Abreu \cite{Abr98} to prove that our Sasakian metrics $g_\bfa$ are extremal. The point is that the only way that $\partial_{g_\bfa}^\# s_{g_\bfa}$ can be holomorphic is that the scalar curvature be an affine function of the moment map coordinates.

}

\end{remark}

\begin{remark}
{\rm The last statement of Theorem \ref{sasex} says that the Sasakian structure $\cals_{1,0}$ is {\it strongly extremal} in the sense of \cite{BGS07b}. There the authors described another variational problem varying over the set of Reeb vector fields that lie in the Lie algebra of a maximal torus.  The critical points of this functional are called {\it strongly extremal Sasakian structures}. In this case $\cals_{1,0}$ is a null eta-Einstein structure with constant $\Phi$-sectional curve $-3$; whereas, the generic Sasakian structures $\cals_{1,\bfa}$ are not eta-Einstein although they are extremal, of course}.
\end{remark}

\section{Compact Quotients}
 
One can easily construct compact nilmanifolds $\gN^{2n+1}_k$  by forming the quotient manifold by the subgroup $\gH^{2n+1}(\bbz,k)$ of $\gH^{2n+1}$ obtained by restricting the coordinates $(z,x_1,\cdots,x_n,y_1,\cdots,y_n)$ in Equation (\ref{Heisgroup}) to take values in the set of all integers divisible by the integer $k> 0.$ It is easy to see that $H_1(\gN^{2n+1}_k,\bbz)=\bbz^{2n}+\bbz_k.$ More generally, we consider a lattice subgroup $\grG$ of $\gN^{2n+1}$ with a compact quotient $\gN^{2N+1}(\grG)=\gH^{2n+1}/\grG$.  Then up to an automorphism of $\gH^{2n+1}$ each nilmanifold $\gN^{2N+1}(\grG)$ is determined by an $n$-tuple $\bfl=(l_1,\cdots,l_n)$ of integers such that $l_i|l_{i+1}$ \cite{Tol78,Fol04}. If we take the right action of $\grG_\bfl=\grG$ of $\gH^{2n+1}$, the quotient inherits the standard right Sasakian structure $(\xi^R,\eta^R,\Phi^R,g^R)$, since it is invariant under this action. The quotient is also a homogeneous space inherited from the left action of $\gH^{2n+1};$ however, this does not leave the right Sasakian structure on $\gH^{2n+1}$ invariant, so the Sasakian structure and homogeneous structure on $\gN^{2n+1}(\grG_\bfl )$ are incompatible \cite{BGO06}. (Of course, one can interchange right and left and obtain an isomorphic model of $\gN^{2n+1}(\grG_\bfl ).$) For any lattice subgroup $\grG_\bfl$ its centralizer in $\gH^{2n+1}$ is the one parameter group generated by the standard Reeb vector field $\xi=\partial_z$, so only the standard Sasakian structure passes to the quotient manifold $\gN^{2n+1}(\grG_\bfl)$.  Indeed up to a finite group the automorphism group $\gA\gu\gt(\cals)$ is one dimensional, and generated by the Reeb vector field. So the Sasaki cone is one dimensional. Summarizing we have

\begin{proposition}\label{compquot}
The only Sasakian structure $\cals_{1,\bfa}$ that passes to a compact quotient is the standard one with $\bfa=0.$
Moreover, the induced Sasakian structure on the quotient $\gN^{2n+1}(\grG_\bfl)$ is null-eta-Einstein with constant $\Phi$-holomorphic sectional curvature $-3,$ and the quotient space $\gN^{2n+1}(\grG_\bfl )/\calf_\xi$ is an Abelian variety $\bbc^n/\grL_\bfl$ determined by the lattice $\grL_\bfl=\pi(\grG_\bfl),$ where $\pi$ is induced by the natural quotient map $\gH^{2n+1}\ra{1.9} \gH^{2n+1}/\calf_\xi\approx \bbc^n.$
\end{proposition}

The procedure can be inverted, that is we can begin with a polarized Abelian variety $(\bbc^n/\grL_\bfl,L)$ where $L$ is a positive line bundle on $\bbc^n/\grL_\bfl$, and construct $\gN^{2n+1}(\grG_\bfl)$ as the total space of the unit circle bundle in the line bundle $L.$ Folland \cite{Fol04} proved that, up to holomorphic CR equivalence there is a 1-1 correspondence between the nilmanifolds $\gN^{2n+1}(\grG_\bfl)$ and polarized Abelian varieties $(\bbc^n/\grL_\bfl,L)$. 

A result of Marinescu and Yeganefar \cite{MaYe07} says that every compact Sasakian manifold is {\it holomorphically fillable} in the sense that it can be realized as the boundary of a compact strictly pseudoconvex complex manifold $V$. In fact, Folland \cite{Fol04} shows how to present the nilmanifolds $\gN^{2n+1}(\grG_\bfl)$ as a boundary of the unit disc bundle $V$ in the dual bundle $L^*$ by using the theta functions that determine $L.$ Since the zero section of $L^*$ is identified with the Abelian variety $(\bbc^n/\grL_\bfl)$ this is not a Stein filling. In fact, Stein fillings of $\gN^{2n+1}(\grG_\bfl)$ may not exist \cite{P-P07c}. It is interesting to compare this situation with that of links of isolated hypersurface singularities by weighted homogeneous polynomials. These are called {\it Milnor fillable} in \cite{P-P07c}. By the well-known Milnor Fibration Theorem $\gN^{2n+1}(\grG_\bfl)$ can be represented as a link of a weighted homogeneous polynomial only if $n=1,$ and in this case only the nilmanifolds $\gN^3_k$ with $k=1,2,3$ can be so represented \cite{BG05}. In dimension three holomorphically fillable implies Stein fillable, so all Sasakian 3-manifolds are Stein fillable.

\def\cprime{$'$} \def\cprime{$'$} \def\cprime{$'$} \def\cprime{$'$}
  \def\cprime{$'$} \def\cprime{$'$} \def\cprime{$'$} \def\cprime{$'$}
  \def\cdprime{$''$}
\providecommand{\bysame}{\leavevmode\hbox to3em{\hrulefill}\thinspace}
\providecommand{\MR}{\relax\ifhmode\unskip\space\fi MR }
\providecommand{\MRhref}[2]{%
  \href{http://www.ams.org/mathscinet-getitem?mr=#1}{#2}
}
\providecommand{\href}[2]{#2}

\end{document}